\documentclass[amssymb,amsfonts,reqno]{amsart}

\usepackage{amssymb} 
\usepackage{amsmath} 
\usepackage{amsthm} 
\usepackage{mathtools} 

\def\be#1{\begin{equation}\label{#1}}
\def\bas{\begin{align*}}
\def\eas{\end{align*}}
\def\bi{\begin{itemize}}
\def\ei{\end{itemize}}

\theoremstyle{plain}
   \newtheorem{theorem}[subsection]{Theorem}

   \newtheorem{remark}[subsection]{Remark}

   \newtheorem{proposition}[subsection]{Proposition}
   
   \newtheorem{lemma}[subsection]{Lemma}

\newcommand{\RR}{\mathbb{R}}

\DeclareMathOperator{\supp}{supp}
\DeclareMathOperator{\diag}{diag}

\author{Reuben Wheeler}
\date{}
\address{Maxwell Institute of Mathematical Sciences and the School of Mathematics, University of
Edinburgh, JCMB, The King's Buildings, Peter Guthrie Tait Road, Edinburgh, EH9 3FD, Scotland}
\email{s1641717@sms.ed.ac.uk}

\title[A note on the $L^p$ integrability of a class B\^ochner-Riesz kernels]{A note on the $L^p$ integrability of a class of B\^ochner-Riesz kernels}
\begin{document}

\maketitle

\begin{abstract} For a general compact variety $\Gamma$ of arbitrary codimension, one can
consider the $L^p$ mapping properties of the B\^ochner-Riesz multiplier 
$$
m_{\Gamma, \alpha}(\zeta) \ = \ {\rm dist}(\zeta, \Gamma)^{\alpha} \phi(\zeta)
$$
where $\alpha > 0$ and $\phi$ is an appropriate smooth cut-off function. Even for the sphere $\Gamma = {\mathbb S}^{N-1}$,
the exact $L^p$ boundedness range remains a central open problem in Euclidean Harmonic Analysis. In this
paper we consider the $L^p$ integrability of the
B\^ochner-Riesz convolution kernel for a particular class of varieties (of any codimension). For a subclass of these varieties the range of $L^p$ integrability of the kernels differs substantially from the $L^p$ boundedness range of the corresponding B\^ochner-Riesz multiplier operator.
\end{abstract}

\section{Introduction}

When one studies Fourier multiplier operators $T_m$, with a compactly supported multiplier $m$, $\widehat{T_m(f)}(\xi)=m(\xi)\hat{f}(\xi)$, a necessary condition for $T_m$ to be bounded on $L^p(\RR^N)$  is that the kernel $K=\check{m}\in L^p$. To see this, one can take a Schwartz function $\phi$ with $\hat{\phi}(\xi)=1$ for $\xi\in \supp m$ so that $T_m \phi = \check{m}$.

In a remarkable paper of Heo, Nazarov, and Seeger, \cite{heoNazSee}, they show that this natural necessary condition is also sufficient for $T_m$ to be bounded on $L^p$ in the range $1<p<\frac{2(N-1)}{N+1}$ 
whenever $m$ is a compactly supported \textit{radial} multiplier. The Radial Multiplier Conjecture states that for radial multipliers $m$, $\check{m}\in L^p$ implies $T_m$ is bounded on $L^p$ holds in the range $1<p<\frac{2N}{N+1}$.

The canonical example of a compactly supported radial multiplier is the classical B\^{o}chner-Riesz multiplier \[m_\alpha(\xi)=\left(1-|\xi|^2\right)_+^{\alpha}.\] The corresponding convolution kernel $K_\alpha=\check{m}_\alpha$ is well known to lie in $L^p$ for \[p>p_{\alpha,N}=\frac{2N}{N+1+2\alpha};\] see \cite{herz}. 
We also have, from the famous work of C. Fefferman on the ball multiplier, \cite{feffBall}, the following necessary condition. For $N\geq 2$, if the multiplier operator $T_{m_{\alpha}}$ is bounded on $L^p$ for $p\neq 2$, then $\alpha>0$. The B\^{o}chner-Riesz Conjecture states that $T_{m_\alpha}$ is bounded on $L^p(\RR^N)$ if and only if \[p_{\alpha,N}<p<p_{\alpha,N}'\mbox{ and }\alpha>0.\]
The restrictions on the parameters may be equivalently expressed as \[\alpha>\max\left\lbrace N\left|\frac{1}{2}-\frac{1}{p}\right|-\frac{1}{2} ,0\right\rbrace.\]

The radial multiplier Conjecture is a significant generalisation 
of the B\^{o}chner-Riesz Conjecture. However, the result in \cite{heoNazSee} gives no new improvements on the B\^{o}chner-Riesz Conjecture.

In this paper, we examine a class of compactly supported multipliers arising as B\^{o}chner-Riesz multipliers associated to a family of embedded varieties, $\Gamma$, in $\RR^N$ of arbitrary codimension. 

We consider a compact neighbourhood about the origin, $\Gamma$, of the $n$-dimensional variety \[\left\lbrace (\xi,\Psi(\xi))\right\rbrace\subset \RR^N,\]
parametrised as the graph of smooth $\Psi:\RR^n\rightarrow \RR^L$. We write $\Psi:\RR^n\rightarrow \RR^L$ component-wise as \[\Psi(\xi)=\left(\psi_1(\xi),\psi_2(\xi),\ldots,\psi_{L}(\xi))\right).\] 

For a general variety $\Gamma$ and some $\alpha>0$,  \[m(\zeta)=d(\zeta,\Gamma)^\alpha\phi(\zeta)\]
is called the B\^ochner-Riesz multiplier with exponent $\alpha>0$. Here 
$\phi$ is some appropriate bump function whose support intersects $\Gamma$.

We may assume that $\Gamma$ is paramaterised as the graph of a smooth function $\Psi$, with $\phi$ supported in a small neighbourhood of $0$ and $\Psi$ chosen such that such that $\Psi(0)=0$ and $\nabla\Psi(0)=0$. This is achieved using a smooth partition of unity, translating and rotating. 
One can reduce the study of the $L^p$ mapping properties of the operators $T_{m}$ defined by the above multipliers to the study of the $L^p$ mapping properties of related multiplier operators $T_{m_{\Gamma,\alpha}}$, whose multipliers are given by \begin{equation}\label{eq:genBRmult}m_{\Gamma,\alpha}(\zeta)=m_{\Gamma,\alpha}(\xi,\eta)=\phi(\xi)|\eta-\Psi(\xi)|^\alpha\chi(\eta-\Psi(\xi)).\end{equation} 
Here $\zeta=(\xi,\eta)\in\RR^{n+L}=\RR^N$, $\phi:\RR^n\rightarrow \RR$ and $\chi:\RR^L\rightarrow \RR$ are appropriate radial bump functions with small supports contained in $B^n(0,\delta)$ and $B^L(0,\delta)$, respectively.
See \cite{mythesis} for details.

For a particular class of surfaces, we determine the precise $L^p(\RR^N)$ integrability range of the corresponding convolution kernel $K_{\Gamma,\alpha}=\check{m}_{\Gamma,\alpha}$. For a subclass of these surfaces we also show that the range of integrability for the convolution kernel is strictly larger than the $L^p(\RR^N)$ boundedness range of the B\^{o}chner-Riesz operator $T_{m_{\Gamma,\alpha}}$. Although compactly supported, these multipliers are not radial multipliers.

Let $\Gamma$ be a compact piece of the graph $\{(\xi, \Psi(\xi))\in\RR^N: \xi \in {\mathbb R}^n\}$ taken in a small
neighbourhood of the origin. Here \begin{equation}\label{eq:psiHom}
\Psi(x) = (a_1 |\xi|^{d_1}, a_2 |\xi|^{d_2}, \ldots, a_L |\xi|^{d_L})\end{equation}
and the $d_j\geq 2$ are distinct. The surface $\Gamma$ has dimension $n$ and codimension $L$. 

\begin{theorem}\label{thm:suffKrad}\label{thm:kLpNecRadial}Consider $\Gamma$ as given as a neighbourhood about the origin of the graph given by \eqref{eq:psiHom} with some coefficient $a_j\neq 0$. Then the convolution kernel $K_{\Gamma,\alpha}=\check{m}_{\Gamma,\alpha}\in L^p(\RR^N)$ if and only if $p>\frac{L+n}{L+\alpha+\frac{n}{2}}$.\end{theorem}

Without loss of generality, we henceforth assume that $a_j\neq 0$, for $1\leq j\leq L_1$, $d_1<d_2<\ldots<d_{L_1}$, and $a_j=0$, for $j>L_1$, for some $L_1\leq L$. We will also make use of the following Lemma, which is proved in Section \ref{sec:prelim}.

\begin{lemma}
\label{lem:ajAreOne}To prove Theorem \ref{thm:suffKrad}, it suffices to assume that the non-zero coefficients $a_j$ are all equal to $1$. In particular, we may assume that $a_1=a_2=\ldots=a_{L_1}=1$.
\end{lemma}

The proofs of necessity and sufficiency in Theorem \ref{thm:suffKrad} are different. The kernel $K_{\Gamma,\alpha}$ is given as an oscillatory integral with an explicit phase. To prove that \begin{equation}\label{eq:necKLp}K_{\Gamma,\alpha}\in L^p\implies p>\frac{L+n}{L+\alpha+\frac{n}{2}}\end{equation} we restrict our attention to a region where the critical point of the phase is non-degenerate and use stationary phase techniques to obtain pointwise estimates on the size of the kernel. To obtain the sufficient condition \begin{equation}\label{eq:suffKLp} p>\frac{L+n}{L+\alpha+\frac{n}{2}}\implies K_{\Gamma,\alpha}\in L^p\end{equation} we adapt methods developed in work by A. Karatsuba, G. I. Arkhipov, and V. Chubarikov, \cite{arkhipovNoGaps}, which were used to settle a problem arising from Tarry's problem in Number Theory; rather than pointwise control on the size of the kernel, we form a dyadic partition of space according to its size.

\begin{remark} Theorem \ref{thm:suffKrad} holds for more general varieties including, for example, the curves of standard type found in \cite{SteinWaingerProblemsInHA}; see \cite{mythesis}.
\end{remark}

We now consider results concerning the $L^p(\RR^N)$ boundedness of the multiplier operator $T_{m_{\Gamma,\alpha}}$. We recall some known results and state a sharp result, which is a consequence of our theorem and a sharp restriction theorem.

 For general varieties $\Gamma$, it is well known that the B\^{o}chner-Riesz problem is connected to another fundamental problem in Euclidean Harmonic Analysis, the Fourier restriction problem, which in this context is to determine the $L^p(\RR^N)\rightarrow L^q(\Gamma,\mu)$ mapping properties of the restriction operator $Rf=\hat{f}|_{\Gamma}$, where $\mu$ denotes surface measure on $\Gamma$. Even in the original setting of $\Gamma =S^{N-1}$, as proposed by Stein in the mid 1960's, the Fourier restriction problem 
is unsolved for $N\geq 3$.

Progress with the B\^{o}chner-Riesz Conjecture has historically paralleled progress with the Restriction Conjecture.  Tao established  that the B\^{o}chner-Riesz Conjecture implies the Restriction Conjecture on the sphere, \cite{taoBRimpliesR}; the table contained therein also outlines some of the parallel progress in these two areas.

It is well known that sharp $L^p$ estimates for B\^{o}chner-Riesz multiplier operators, $T_{m_{\Gamma,\alpha}}$, follow from $L^2$ Fourier restriction estimates for $\Gamma$;
\begin{equation}\label{eq:l2restIneq}\left(\int\left|\hat{f}(\xi,\Psi(\xi))\right|^2\phi(\xi)d\xi\right)^\frac{1}{2}\leq C\|f\|_{L^q(\RR^N)}.\end{equation}
This implication for sharp estimates and surrounding ideas date back to Fefferman's thesis, \cite{feffActa}, where the analysis is for the sphere $\Gamma=S^{N-1}$.

Recall that $m_{\Gamma,\alpha}$ was defined only for $\alpha>0$. For varieties $\Gamma$ of arbitrary codimension we have the following result of G. Mockenhoupt, \cite{mockBR}:
\begin{theorem}\label{thm:rest2BR} Let $\Gamma$ be a compact piece near the origin of the graph
 \[\lbrace \left(\xi, \Psi(\xi)\right); \xi \in \RR^n\rbrace\subset\RR^n\times \RR^{L}.\]
 Suppose that the restriction inequality \eqref{eq:l2restIneq} holds. 
Then, for $1\leq p\leq q$ or $q'\leq p <\infty$, the  B\^{o}chner-Riesz multiplier $m_{\Gamma,\alpha}$ defined in \eqref{eq:genBRmult}
defines a multiplier operator $T_{m_{\Gamma,\alpha}}$ which is bounded on $L^p$ for \[\alpha>\max\left\lbrace(n+L)\left|\frac{1}{p}-\frac{1}{2}\right|-\frac{L}{2},0\right\rbrace.\]\end{theorem}

In particular, when $p\leq q<2$, we see that $T_{m_{\Gamma,\alpha}}$ is bounded on $L^p(\RR^N)$ if $p>\frac{L+n}{L+\alpha+\frac{n}{2}}$, which is precisely the $L^p$ integrability range for the B\^{o}chner-Riesz kernels $K_{\Gamma,\alpha}$ considered in Theorem \ref{thm:suffKrad}. Hence, for the varieties considered in Theorem \ref{thm:suffKrad}, we obtain sharp $L^p$ estimates for $T_{m_{\Gamma,\alpha}}$, provided the Fourier restriction inequality \eqref{eq:l2restIneq} holds.

In this paper, we observe that there is a class of surfaces for which the range of $p$ such that the B\^{o}chner-Riesz kernel $K_{m_{\Gamma,\alpha}}=\check{m}_{\Gamma,\alpha}\in L^p$ differs from the $L^p$ boundedness range for $T_{m_{\Gamma,\alpha}}$. These examples are instances of varieties $\Gamma$ given as a compact neighbourhood about the origin of \[ \left\lbrace\left(\xi,\widetilde{\Psi}(\xi),0\right)\right\rbrace,\] with $\widetilde{\Psi}:\RR^n\mapsto \RR^{L_1}$. These lie in a hyperplane in $\RR^N$ if $L_1<L$, i.e. if there are some $0$ components of the graphing function. In this case it is well known that the Fourier restriction inequality \eqref{eq:l2restIneq} can not hold for any $1<q$. 
Nevertheless, a restriction inequality may hold within the ambient hyperplane $\RR^{n+L_1}$. Indeed, let us write \[ \left\lbrace(\xi,\Psi(\xi))\in\RR^{n}\times\RR^{L}\right\rbrace=\left\lbrace(\xi,\widetilde{\Psi}(\xi),0)\in\RR^{n}\times\RR^{L_1}\times\RR^{L-L_1}\right\rbrace,\]
where $\widetilde{\Psi}:\RR^n\rightarrow \RR^{L_1}$ is smooth. We suppose \begin{equation}\label{eq:l2restrIneq2}\left(\int_{\RR^n}|\hat{g}(\xi,\widetilde{\Psi}(\xi))|^2\phi(\xi)d\xi\right)^\frac{1}{2}\leq C\|g\|_{L^q(\RR^{n+L_1})}\end{equation}
holds for some $1<q<2$ and some constant $C$. Then we have the following generalisation of Theorem \ref{thm:rest2BR}.

\begin{theorem}\label{thm:rest2BRsubs}Let $\Gamma$ be a compact piece near the origin of the graph
\[\left\lbrace \left(\xi, \widetilde{\Psi}(\xi),0\right); \xi \in \RR^n\right\rbrace\subset  \RR^n\times  \RR^{L_1}\times\RR^{L_2}\] and $\Gamma'$ be the corresponding projection to \[\left\lbrace \left(\xi, \widetilde{\Psi}(\xi)\right); \xi \in \RR^n\right\rbrace\subset\RR^n\times  \RR^{L_1}.\]
Suppose that the $L^{q}\left(\RR^{n+L_1}\right)\rightarrow L^2(\Gamma, \mu)$ restriction inequality \eqref{eq:l2restrIneq2} holds. Then, for $1\leq p\leq q$ or $q'\leq p <\infty$, the  B\^{o}chner-Riesz multiplier $m_{\Gamma,\alpha}$ given by \eqref{eq:genBRmult} defines a multiplier operator $T_{m_{\Gamma,\alpha}}$ which is bounded on $L^p$ if \[\alpha>\max\left\lbrace(n+L_1)\left|\frac{1}{p}-\frac{1}{2}\right|-\frac{L_1}{2},0\right\rbrace.\]
\end{theorem}

The proof of Theorem \ref{thm:rest2BRsubs} follows along known lines of argument (those which establish Theorem \ref{thm:rest2BR}); see \cite{mythesis}.

Let $\Gamma$ be a compact piece about the origin of the surface $\left\lbrace (\xi,\Psi(\xi))\right\rbrace$, with $\Psi$ given by \eqref{eq:psiHom} with $a_{1},a_2,\ldots,a_{L_1}\neq 0$, $d_1<d_2<\ldots<d_{L_1}$, and $L_1<L$. Then a necessary condition for the restriction inequality \eqref{eq:l2restrIneq2} to hold is \begin{equation}\label{eq:KnappCond}\frac{q'}{2}\geq 1+\frac{D}{n},\end{equation}
where $D=\sum_{j=1}^{L_1} d_j$.
This follows from a standard Knapp example. The following result has been established by J. Wright, \cite{wrightPrivCom}.
\begin{proposition}\label{prop:resthold}
With $\Gamma$ as above, if $d_1\geq n(L_1+1)$, then \eqref{eq:l2restrIneq2} holds if $\frac{q'}{2}\geq 1+\frac{D}{n}$.
\end{proposition}

Proposition \ref{prop:resthold} gives us a class of varieties of arbitrary codimension where the optimal $L^2$ Fourier restriction estimate holds. Combining this with Theorem \ref{thm:rest2BRsubs}, de Leuuw's Multiplier Theorem, duality, and Theorem \ref{thm:suffKrad} we establish the following sharp result; see Section \ref{sec:prelim}.

\begin{proposition}\label{prop:deLeuuwImp}Let $\Psi$ be given by \eqref{eq:psiHom} with $L_1<L$, $d_1<d_2,<\ldots,<d_{L_1}$, and $d_1\geq n(L_1+1)$. Let $\Gamma$ be a compact piece about the origin of the graph \[\left\lbrace(\xi,\Psi(\xi))\right\rbrace,\] and $\frac{q'}{2}= 1+\frac{D}{n}$. For $1\leq p \leq q$ or $q'\leq p <\infty$, $T_{m_{\Gamma,\alpha}}$ is bounded on $L^p$ if and only if
$\frac{L_1+n}{L_1+\alpha+\frac{n}{2}}<p<\frac{L_1+n}{L_1-\alpha+\frac{n}{2}}$.
\end{proposition}

One expects these propositions to hold without the extra condition $d_1\geq n(L_1+1)$. For the curves described by the equation \eqref{eq:psiHom} when $n=1$ and $L_1=L$ (i.e. there are no vanishing coefficients) this is indeed the case; see \cite{DMRestrDegenCurves}.

In distinction to proposition \ref{prop:deLeuuwImp}, Theorem \ref{thm:suffKrad} implies that the convolution kernel $K_{\Gamma,\alpha}=\check{m}_{\Gamma,\alpha}$ lies in $L^p(\RR^N)$ if and only if $p>\frac{L+n}{L+\alpha+\frac{n}{2}}$. It is easily seen for $\alpha<\frac{n}{2}$ that $\frac{L_1+n}{L_1+\alpha+\frac{n}{2}}\geq \frac{L+n}{L+\alpha+\frac{n}{2}}$ if $L_1\leq L$, with strict inequality if $L_1\neq L$.
Hence, for the varieties in Theorem \ref{thm:suffKrad} with $L_1<L$, the $L^p$ boundedness range for $T_{m_{\Gamma,\alpha}}$ differs from the $L^p$ integrability range of $K_{\Gamma,\alpha}$.

Regarding the Fourier restriction problem for (a neighbourhood of the origin of) the moment curve, \[(t,t^2,t^3,\ldots,t^d),\] there is the initial result of Drury \cite{drury}, which tells us that the restriction inequality \eqref{eq:l2restIneq} holds for $q'\geq d(d+1)$. In the field, it was not known that this result was sharp until the discovery of \cite{arkhipovNoGaps}, where the necessary condition $q'\geq d(d+1)$ had been established. As mentioned previously, it is a method from the same paper that we adapt to show the implication \eqref{eq:suffKLp} of Theorem \ref{thm:suffKrad}: we form a dyadic partition of space according to the size of the kernel.

\textbf{Acknowledgements.} 
The author was supported by The Maxwell Institute Graduate School in Analysis and its Applications, a Centre for Doctoral Training funded by the UK Engineering and Physical Sciences Research Council (Grant EP/L016508/01), the Scottish Funding Council, Heriot-Watt University and the University of Edinburgh.

The author would like to thank Jim Wright for the introduction of the problem and for many helpful suggestions in the development of this paper. The author would also like to think Aswin Govindan-Sheri for a careful reading of the manuscript.

\section{Outline of paper}

Section \ref{sec:not} introduces notation and derives the B\^ochner-Riesz kernels. We also observe some basic properties of the kernels.

In Section \ref{sec:prelim} we show how Proposition \ref{prop:deLeuuwImp} follows from Theorem \ref{thm:rest2BRsubs} and Theorem \ref{thm:suffKrad} applied with reference to de Leeuw's Theorem. We also prove Lemma \ref{lem:ajAreOne}.

In Sections \ref{sec:suffKLp} and \ref{sec:necMinMp} we prove Theorem \ref{thm:suffKrad}. In Section \ref{sec:suffKLp}  we prove the sufficient condition for $K_{\Gamma,\alpha}\in L^p$, \eqref{eq:suffKLp}. In Section \ref{sec:necMinMp} the reverse implication, \eqref{eq:necKLp}, is proved, this is the necessary condition for $K_{\Gamma,\alpha}\in L^p$.

\section{Notation}\label{sec:not}
Throughout this paper $C$ will be used to denote a constant, its value may change from line to line. We use the notation $X\lesssim Y$ or $Y\gtrsim X$ if there exists some implicit constant $C$ such that $X\leq CY$. When we wish to highlight the dependence of the implied constant $C$ on some other parameter, say $C=C(M)$, we will use the notation $X\lesssim_M Y$. We use the notation $X\ll Y$ or $Y\gg X$ if there exists some suitable large constant $D$ such that $DX\leq Y$.

We denote by $B^m(x,r)$ the Euclidean ball of radius $r$ centred at $x$ in $\RR^m$, or simply $B(x,r)$ if the dimension is clear. 

We use the notation \[\mathcal{F}\left(f\right)(\xi)=\hat{f}(\xi)=\int e^{-2\pi i \xi \cdot x}f(x)dx,\] for the Fourier transform and likewise for the inverse Fourier transform we write \[\mathcal{F}^{-1}\left(g\right)(x)=\check{g}(x)=\int e^{2\pi i x \cdot \xi}g(\xi)d\xi.\]

We also denote by $\hat{\sigma}$ the Fourier transform of the surface measure of the sphere $S^{m-1}$:\[\hat{\sigma}(x)=\int_{S^{m-1}}e^{2\pi i x\cdot \omega}d\sigma(\omega).\]

For $p\in (1,\infty)$ we denote by $\mathcal{M}_p$ the space of Borel measurable functions $m$ for which the multiplier operator defined a priori by
$$f\mapsto \left(m\cdot \hat{f}\right)^{\check{}}$$
is bounded from $L^p$ to $L^p$.

Recall that we write $L=L_1+L_2$. Typically, in what follows $x,y$, and $z$ will denote points in $\RR^n$, $\RR^{L_1}$, and $\RR^{L_2}$, respectively. Since we may have that $L_2=0$ but integrate with respect to $dz$ we follow the convention that, when $L_2=0$, $\RR^{L_2}=\RR^{0}=\lbrace 0\rbrace$ and $dz$ is the counting measure. Similarly for $S^{0}$, which we regard as the set $\lbrace -1,1\rbrace$ equipped with the counting measure.

\label{sec:genBRmult}
All multiplier operators are expressible as convolution operators. 
We turn our sights to the convolution kernel of $T_m$, $K_{\Gamma,\alpha}=\check{m}$, where $m=m_{\Gamma,\alpha}$, as given by \eqref{eq:genBRmult}. Observe that \[K_{\Gamma,\alpha}(x,y,z)=\int_{\RR^L}\int_{\RR^n}e^{2\pi i \left(x\cdot \xi +(y,z)\cdot \eta\right)}m(\xi,\eta)d\xi d\eta\]
\[=\int_{\RR^L}\int_{\RR^n} e^{2\pi i \left(x\cdot \xi +(y,z)\cdot \eta\right)}\phi(\xi)|\eta-\Psi(\xi)|^\alpha\chi(\eta-\Psi(\xi))d\xi d\eta\]
\[=\int_{\RR^L}\int_{\RR^n} e^{2\pi i \left(x\cdot \xi +(y,z)\cdot (\eta+\Psi(\xi))\right)}\phi(\xi)|\eta|^\alpha\chi(\eta)d\xi d\eta\]
\begin{equation}\label{eq:stupidProd}=A_{\alpha}(y,z)\left(\int_{\RR^n} e^{2\pi i \left(x\cdot \xi +(y,z)\cdot \Psi(\xi)\right)}\phi(\xi)d\xi\right) \end{equation}
\begin{equation}\label{eq:kerProd}=A_\alpha(y,z)k(x,y),\end{equation}
where \[k(x,y)=\int_{\RR^n} e^{2\pi i \left(x\cdot \xi +y\cdot \widetilde{\Psi}(\xi)\right)}\phi(\xi)d\xi\mbox{ and }A_\alpha(y,z)=\int_{\RR^L} e^{2\pi i (y,z)\cdot \eta}|\eta|^\alpha\chi(\eta)d\eta.\]
Unlike $A_\alpha$, obtaining pointwise control $k$ is a delicate matter. Such pointwise control will only be required when we prove the necessary condition \eqref{eq:necKLp} in Section \ref{sec:necMinMp}, where we establish a lower bound on $|k|$ in a restricted region of $\RR^{n+L_1}$.

We now state bounds on the size of the function $A_\alpha$. 
\begin{lemma}\label{lem:Abound}\label{lem:Aasymptotic}For large $|(y,z)|$, there exist constants $0<c_\alpha\leq C_\alpha$ such that \begin{equation}\label{eq:AalphBnd}c_\alpha|(y,z)|^{-L-\alpha}\leq|A_\alpha((y,z))|\leq C_\alpha|(y,z)|^{-L-\alpha}.\end{equation}
\end{lemma}
This is routine to establish and we omit the proof, it is presented in \cite{mythesis}.

\section{Preliminary Results}\label{sec:prelim}
\begin{proof}[Proof of Proposition \ref{prop:deLeuuwImp}]
By Proposition \ref{prop:resthold} we see that $L^p\rightarrow L^2(\Gamma',\mu)$ restriction holds for $1\leq p\leq q$. Therefore, we can apply Theorem \ref{thm:rest2BRsubs} to see that, where $1\leq p \leq q$ or $q'\leq p <\infty$, $T_{m_{\Gamma,\alpha}}$ is bounded on $L^p$ for $\frac{L_1+n}{L_1+\alpha+\frac{n}{2}}<p<\frac{L_1+n}{L_1-\alpha+\frac{n}{2}}$. 
Instead of considering the multiplier \[m_{\Gamma,\alpha}(\xi,\eta,\lambda)=|(\eta,\lambda)-\Psi(\xi)|^\alpha\chi\left((\eta,\lambda)-\Psi(\xi)\right)\phi(\xi),\] we restrict this continuous multiplier to $V=\RR^{n+L_1}$, corresponding to $\lambda=0$. de Leeuw's Theorem, \cite{deLeeuw}, tells us that if $m_{\Gamma,\alpha}\in\mathcal{M}_p(\RR^N)$, then $m_{\Gamma,\alpha}|_V\in \mathcal{M}_p(\RR^{n+L_1})$. This is the B\^{o}chner-Riesz multiplier on $\RR^{n+L_1}$ corresponding to the compact neighbourhood $\Gamma'$ about the origin of the surface $\left\lbrace (\xi,\widetilde{\Psi}(\xi))\right\rbrace$. We apply Theorem \ref{thm:kLpNecRadial} to see that the corresponding kernel is in $L^p$ only for $p>\frac{L_1+n}{L_1+\alpha+\frac{n}{2}}$. By duality, we must also have $p<\frac{L_1+n}{L_1-\alpha+\frac{n}{2}}$ for $T_{m_{\Gamma,\alpha}}$ to be bounded on $L^p$.
\end{proof}

In the following proof of Lemma \ref{lem:ajAreOne} we use the notation $(x,y)\in \RR^n\times \RR^L$. It is this proof and the arguments contained therein that justify the notation we will use thereafter of $(x,y,z)\in \RR^n\times\RR^{L_1}\times\RR^{L_2}$. With the notation $(x,y)\in \RR^n\times \RR^L$, we write \eqref{eq:stupidProd} as
\begin{equation}\label{eq:stupidProd2}K_{\Gamma,\alpha}(x,y)=A_{\alpha}(y)\int_{\RR^n} e^{2\pi i \left(x\cdot \xi +y\cdot \Psi(\xi)\right)}\phi(\xi)d\xi .\end{equation}
\begin{proof}[Proof of Lemma \ref{lem:ajAreOne}]
Let $\Psi(\xi)=\left(a_1|\xi|^{d_1},a_2|\xi|^{d_2},\ldots,a_L|\xi|^{d_L}\right)$ be given as in \eqref{eq:psiHom}, where $a_j\neq 0$ for $1\leq j \leq L_1$ and $d_1<d_2<\ldots<d_{L_1}$. We have that $K_{\Gamma,\alpha}(x,y)=A_\alpha(y)k(x,y)$, where \[A_{\alpha}(y)=\int_{\RR^L}e^{2\pi i y\cdot\eta}|\eta|^\alpha\chi(\eta)d\eta\]
and
\[k(x,y)=\int_{\RR^n}e^{2\pi i \left(x\cdot \xi + y\cdot \Psi(\xi)\right)}\phi(\xi)d\xi.\]

We set $\tilde{y}=My$ in the $y$ integral defining $k$, where $M=\diag(a_1,a_2,\ldots,a_{L_1},1,1,\ldots,1)$ is a diagonal matrix, to see that \[k(x,y)=\tilde{k}(x,My)\]where \[\tilde{k}(x,y)=\int_{\RR^n}e^{2\pi i \left(x\cdot \xi + \sum_{j=1}^{L_1}y_j|\xi|^{d_j}\right)}\phi(\xi)d\xi.\]

We denote by $\widetilde{S_0}$ a region with large $|\tilde{y}|\sim |M^{-1}\tilde{y}|$ where we can use the comparison $|A(\tilde{y})|\sim |\tilde{y}|^{-L-\alpha}\sim |M^{-1}\tilde{y}|^{-L-\alpha}\sim|A(M^{-1}\tilde{y})|=|A(y)|$. Let $S_0$ be the corresponding region under the change of variables $(x,\tilde{y})\mapsto (x,M^{-1}\tilde{y})$. We can make the change of variables $\tilde{y}=My$ in the $L^p$ integration of $K_{\Gamma,\alpha}(x,y)=k(x,y)A(y)$. The change of variables has a constant Jacobian. Using the comparison, it is simply verified that \[K_{\Gamma,\alpha}\in L^p(S_0) \iff K_{\widetilde{\Gamma},\alpha}\left(\widetilde{S_0}\right),\]
where $\widetilde{\Gamma}$ is a suitable small neighbourhood of the graph \[\left\lbrace (\xi,|\xi|^{d_1},|\xi|^{d_2},\ldots,|\xi|^{d_{L_1}},0,\ldots,0)\right\rbrace.\] It remains to consider mutual $L^p$ boundedness on the complements: with $S_1=S_0^c$ and $\widetilde{S_1}$ its image under the change of variables $\tilde{y} =My$ in the $y$ coordinate, we must show $K_{\Gamma,\alpha}\in L^p(S_1)\iff K_{\widetilde{\Gamma},\alpha}(\widetilde{S_1})$. In fact, one may routinely verify that $K_{\Gamma,\alpha}\in L^p(S_1)$ for all $1\leq p\leq \infty$ using an $L^\infty$ estimate on $K_{\Gamma,\alpha}$ and the method of non-stationary phase. Likewise, one may verify $K_{\Gamma,\alpha}\in L^p(\widetilde{S_1})$ for all $1\leq p\leq \infty$.

\end{proof}

\section{Sufficency of Theorem \ref{thm:suffKrad}; proof of \eqref{eq:suffKLp}}
\label{sec:suffKLp}
In this section we prove one part of Theorem \ref{thm:suffKrad}, the implication \eqref{eq:suffKLp}. That is the sufficient condition for the kernel $K$, \eqref{eq:kerProd}, to be in $L^p$. According with Lemma \ref{lem:ajAreOne}, we consider those $\Psi$, \eqref{eq:psiHom}, 
where we have\begin{equation}\label{eq:ensuringAjNice} a_j=1\mbox{, for }j\leq L_1,\,\ a_j=0,\mbox{ for }j>L_1\mbox{, and }d_1<d_2<\ldots<d_L.\end{equation}

We use the following lemma, this van de Corput estimate is essential to the proof of our main result. It is a simple corollary of the van de Corput estimates proved in \cite{arkhipovTrigSums}, the corollary may be proved exactly as for the corresponding classical van de Corput estimates for oscillatory integrals with smooth amplitude.

\begin{lemma}\label{lem:mixedVdc}Let $\Phi:\RR\rightarrow \RR$ be a polynomial of degree $d$ and $\phi$ be a smooth function with $\supp\phi\subset (a,b)$. Given the bound $\inf_{t\in [a,b]}\sum_{j=1}^{d}|\Phi^{(j)}(t)|^\frac{1}{j}\geq \kappa$ we have that \[\left|\int_a^b e^{2\pi i\Phi(t)}\phi(t)dt\right|\lesssim_{d}\min\left\lbrace (b-a),\kappa^{-1}\right\rbrace\|\phi'\|_{L^1}.\]In particular, we have that \[\left|\int_a^b e^{2\pi i\Phi(t)}\phi(t)dt\right|\lesssim\min\left\lbrace (b-a),\kappa^{-1}\right\rbrace(b-a)\|\phi'\|_{L^\infty}.\]\end{lemma}

We now proceed with the analysis of our surfaces. Henceforth, until stated otherwise, we denote by $\Phi(t)$
the phase appearing in certain one dimensional oscillatory integrals $I_{\pm}(x,y)=\int e^{2\pi i\Phi(t)}\phi_0(t)dt$. These phases are \[\Phi(t)=\Phi_{x,y,\pm}(t)=\pm |x|t+\sum_{j=1}^{L_1} y_j\psi_j(t),\]
with the $\psi_j(t)=|t|^{d_i}$. For $t\in[-\delta,\delta]$, we denote by $H_t(x,y)$ and $H(x,y)$ the quantities \[H_t(x,y)=\sum_{j=1}^{d_{L_1}}\left|\Phi^{(j)}(t)\right|^\frac{1}{j}\mbox{ and }H(x,y)=\inf_{t\in[-\delta,\delta]}H_t(x,y).\]

We will make use the following lemma.
\begin{lemma}\label{lem:HlowerBnd}
For $|y|\gtrsim 1$, $H(x,y)\gtrsim |y|^\frac{1}{d_{L_1}}$.
\end{lemma}
\begin{proof}
We relate $H$ to a homogeneous $\widetilde{H^1}$, where \[\widetilde{H^{1}}(x,y)=\max_{j=1,2,\ldots,d_{L_1}}\inf_{t\in [-\delta,\delta]}\left|\Phi_{x,y}^{(j)}(t)\right|.\]
We show that this, which is clearly homogeneous of degree $1$, is bounded positively away from $0$ on the sphere. We first show that it is non-vanishing on the sphere. 

Suppose that $\widetilde{H^{1}}(x,y)=0$. Then, in particular, $|y_{L_1}|\sim \inf_{t\in [a,b]}\left|\Phi^{(d_{L_1})}_{x,y}(t)\right|=0$ so that $y_{L_1}=0$. Considering the $d_{L_1-1}$th derivative in turn we find that $|y_{L_1-1}|\sim \inf_{t\in[-\delta,\delta]}\left|\Phi^{(d_{L_1-1})}\right|=0$. Continuing in this fashion shows that, if $\widetilde{H^{1}}(x,y)=0$, then $(x,y)=0$.

To show that $\inf_{(x,y)\in S^{N-1}}\widetilde{H^1}(x,y)>0$ we suppose by way of contradiction that there exists a sequence $z_n\in S^{N-1}$ with $\widetilde{H^1}(z_n)\rightarrow 0$. Passing to a subsequence if necessary we can assume that $z_n\rightarrow z\in S^{N-1}$. We see that \[\widetilde{H^1}(z_n)\geq\inf_{t\in I}\left|\Phi_{z_n}^{(j)}(t)\right|=\left|\Phi_{z_n}^{(j)}(t_n)\right|\geq \left|\Phi_{z}^{(j)}(t_n)\right|-\left|\Phi_{z}^{(j)}(t_n)-\Phi_{z_n}^{(j)}(t_n)\right|\]\[\geq \inf_{t\in I}\left|\Phi_{z}^{(j)}(t)\right|-C|z_n-z|.\]
Taking the maximum over $j$ shows that \[\widetilde{H^1}(z_n)\geq \widetilde{H^1}(z)-C|z_n-z|.\]
We know that $\widetilde{H^1}(z)\neq 0$ and $|z_n-z|\rightarrow 0$ so taking $n\rightarrow\infty$ we obtain a contradiction, $0>0$. Therefore we can conclude $\inf_{(x,y)\in S^{N-1}}\widetilde{H^1}(x,y)>0$ (in fact, using this, it is possible to show that $\widetilde{H^1}$ is continuous but this is not necessary for our proof). 
As a consequence of this and homogeneity we have that 
$\widetilde{H^{1}}(x,y)\gtrsim |(x,y)|\geq |y|$.

Now relating $H$ to the homogenous $\widetilde{H^1}$
, which is non-vanishing on the sphere, shows that $H(x,y)\gtrsim |y|^\frac{1}{d_{L_1}}$. Indeed, if $\widetilde{H^1}(x,y)\gtrsim 1$, which we have if $|y|\gtrsim 1$, then there exists $t^*=t^*(x,y)\in I$ with $H(x,y)=H_{t^*}(x,y)$. By an appropriate choice of $j_0$ we see that \[\left|y\right|^\frac{1}{d_{L_1}}\leq \left|(x,y)\right|^\frac{1}{d_{L_1}}\lesssim\left(\widetilde{H^1}(x,y)\right)^\frac{1}{d_{L_1}}
\lesssim\left(\widetilde{H^1}(x,y)\right)^\frac{1}{j_0}=\left(\inf_{t\in I}\left|\Phi^{(j_0)}(t)\right|\right)^\frac{1}{j_0}\]\[\leq\left|\Phi^{(j_0)}(t^*)\right|^\frac{1}{j_0}=\max_{j=1,\ldots d_{L_1}}\left|\Phi^{(j)}(t^*)\right|^\frac{1}{j}\leq\sum_{j=1}^{d_{L_1}}\left|\Phi^{(j)}(t^*)\right|^\frac{1}{j}=H_{t^*}(x,y)=H(x,y).\]
\end{proof}

We require the use of an asymptotic expansion before applying the van de Corput estimate. As such, we introduce a cutoff function $\eta$, chosen such that $\eta(r)=0$ for $r\leq \frac{1}{2}$ and $\eta(r)=1$ for $r\geq 1$.

\begin{lemma}\label{lem:vDCrad}
With $\Phi(r)=\Phi_{x,y,\pm}(r)=\pm |x|r+\sum_{j=1}^{L_1} y_j\psi_j(r)$. We define $H(x,y)=\inf_{r\in[-\delta,\delta]}\sum_{j=1}^{d_{L_1}}\left|\Phi^{(j)}(r)\right|^\frac{1}{j}$. If $H(x,y)\geq \kappa$, then we have that \[\left|\int e^{2\pi i\Phi(r)}\frac{1}{|rx|^{\frac{n-1}{2}}}\eta(|x|r)\phi_0(r)r^{n-1}dr\right|\lesssim \kappa^{-1}\frac{1}{|x|^\frac{n-1}{2}}.\]
\end{lemma}
\begin{proof}
When we apply Lemma \ref{lem:mixedVdc} 
we see that, if $H(x,y)\geq \kappa$, then \[\left|\int_0^\infty e^{2\pi i\left(\pm r|x|+ \sum_{i=1}^{L_1}y_i\psi_i(r)\right)} \frac{1}{|rx|^{\frac{n-1}{2}}}\eta(|x|r)\phi_0(r)r^{n-1}dr\right|\lesssim_{d} \|\tilde{\phi}'\|_{L^1}\min\left\lbrace1,\kappa^{-1}\right\rbrace,\]
where $\tilde{\phi}(r)=\frac{1}{|rx|^{\frac{n-1}{2}}}\eta(|x|r)\phi_0(r)r^{n-1}=\frac{1}{|x|^{\frac{n-1}{2}}}\eta(|x|r)\phi_0(r)r^{\frac{n-1}{2}}$. We see that \[\tilde{\phi}'(r)=\frac{1}{|x|^\frac{n-1}{2}}\left(|x|\eta'(|x|r)\phi_0(r)r^{\frac{n-1}{2}}+\eta(|x|r)\phi_0'(r)r^{\frac{n-1}{2}}+\frac{n-1}{2}\eta(|x|r)\phi_0(r)r^{\frac{n-3}{2}}\right)\]
so that \[\|\tilde{\phi}'\|_{L^1}\leq\frac{1}{|x|^\frac{n-1}{2}}\left(\int\left||x|\eta'(|x|r)\phi_0(r)r^{\frac{n-1}{2}}\right|dr+ \int\left|\eta(|x|r)\phi_0'(r)r^{\frac{n-1}{2}}\right|dr\right)\]
\[+ \frac{1}{|x|^\frac{n-1}{2}}\int\left|\frac{n-1}{2}\eta(|x|r)\phi_0(r)r^{\frac{n-3}{2}}\right|dr\]\[\lesssim \frac{1}{|x|^\frac{n-1}{2}} \left(|x| |x|^{-\frac{n-1}{2}}|x|^{-1}+1+\left(1+|x|^{-\frac{n-1}{2}}\right)\right)\sim \frac{1}{|x|^\frac{n-1}{2}}.\]
\end{proof}

\begin{proof}[Proof of \eqref{eq:suffKLp}]
We prove that $k(x,y)(1+|(y,z)|)^{-L-\alpha}\in L^p(\RR^N,dxdydz)$ which, combined with bounds on $A_\alpha$, \eqref{eq:AalphBnd} and $|A_\alpha|\leq C$, gives the desired result. We show the result for $p\leq 2$, the full range of $p$ follows because $L^p\cap L^\infty\subset L^q$ for all $q\in [2,\infty]$. We will use the asymptotic expansion for the surface measure of the sphere.

We first establish the regions where the kernel $k(x,y)(1+|(y,z)|)^{-(L+\alpha)}$ contributes what can be considered error terms in the $L^p$ integration. The first of these is the region $R_{0}$, where $|x|\gg |y|$. The next region is $R_{-1}$, where $|x|\lesssim 1,|y|$. The remaining region, where $1\lesssim |x|\lesssim |y|$ and $|y|\gg 1$, is the main region $R$. Trivially, we always have the estimate $|k(x,y)|\leq C$.

In the region $R_{0}$, we consider the phase function \[\widetilde{\Phi_{x,y}}(t)=\frac{1}{|x|}\left(x\cdot t + y\cdot \Psi(t)\right),\] for which $\|\widetilde{\Phi_{x,y}}\|_{C^{M+1}(\overline{B(0,\delta)})}\lesssim_{M} 1$ and $\left|\nabla \widetilde{\Phi_{x,y}}(t)\right|\gtrsim 1$, for $|t|\leq\delta$. Therefore, by the non-Stationary Phase Lemma, \[\left|k(x,y)\right|=\left|\int e^{2\pi i |x|\widetilde{\Phi_{x,y}}(t)}\phi(t)dt\right|\lesssim C_M\frac{1}{|x|^M}.\] Integrating the resulting estimate we see \[\int\int\int_{R_{0}}|k(x,y)|^p(1+|(y,z)|)^{-p(L+\alpha)}dxdydz\]\[\lesssim  \int\int\int_{R_{0}}\frac{1}{(1+|x|)^{Mp}}\frac{1}{(1+|(y,z)|)^{p(L+\alpha)}}dxdydz,\] which is finite if, for example, $M>n$.

Next, we see for $R_{-1}$ that \[\int\int\int_{R_{-1}}|k(x,y)|^p(1+|(y,z)|)^{-p(L+\alpha)}dxdydz\]\[\lesssim \int_{|x|\lesssim 1}1dx\int\int(1+|(y,z)|)^{-p(L+\alpha)}dydz<\infty.\]

Using polar integration and the fact that $\Psi$ is radial with $\Psi(t)=\Psi_0(|t|)$ and $\phi(t)=\phi_0(|t|)$ we expand \[k(x,y)=\int e^{2\pi i(x\cdot t + y\cdot \Psi(t))}\phi(t)dt\]
 \[=\int\int e^{2\pi i(x\cdot r\omega+y\cdot \Psi_0(r))}\phi_0(r) d\sigma(\omega)r^{n-1}dr\]
\[=\int e^{2\pi i(y\cdot \Psi_0(r))}\int e^{2\pi i(rx\cdot \omega)} d\sigma(\omega)\phi_0(r)r^{n-1}dr\]
\[=\int e^{2\pi i(y\cdot \Psi_0(r))}\hat{\sigma}(rx)\phi_0(r)r^{n-1}dr.\]

For $n=1$ we have that $\hat{\sigma}(rx)=e^{2\pi ir|x|}+e^{-2\pi ir|x|}$. For $n>1$ we use the asymptotic expansion for $\hat{\sigma}$
\[\hat{\sigma}(rx)=\frac{1}{|rx|^{\frac{n-1}{2}}}\left(be^{2\pi ir|x|}+ae^{-2\pi ir|x|} +R_{0,\sigma}(rx)\right),\] for $|rx|\gtrsim 1$, where the $a$ and $b$ are constants and $|R_{0,\sigma}(rx)|\leq C|rx|^{-1}$; see, for example, Chapter 8 of \cite{bigStein}. To make use of this we introduce the cut-off function in $r$, $\eta$. We choose $\eta$ such that $\eta(r)=0$ for $r\leq \frac{1}{2}$ and $\eta(r)=1$ for $r\geq 1$. We see
\[k(x,y)\]
\[= E_1(x,y)+I(x,y),\] 
where \[ E_1(x,y)=\int_0^\infty e^{2\pi i\sum_{i=1}^Ly_i\psi_i(r)} \widehat{\sigma}(rx)(1-\eta(|x|r))\phi_0(r)r^{n-1}dr\]
and
\[I(x,y)=\int_0^\infty e^{2\pi i \sum_{i=1}^Ly_i\psi_i(r)} \widehat{\sigma}(rx)\eta(|x|r)\phi_0(r)r^{n-1}dr.\]

Let us look first to $I(x,y)$. Using the asymptotic expansion, we see that \[I(x,y)\]
\[= I_{+}(x,y)+I_{-}(x,y)+E_2(x,y),\]
where \[I_{+}(x,y)=b\int_0^\infty e^{2\pi i\left(r|x|+ \sum_{i=1}^Ly_i\psi_i(r)\right)} \frac{1}{|rx|^{\frac{n-1}{2}}}\eta(|x|r)\phi_0(r)r^{n-1}dr,\]
\[I_{-}(x,y)=a\int_0^\infty e^{2\pi i\left(-r|x|+ \sum_{i=1}^Ly_i\psi_i(r)\right)} \frac{1}{|rx|^{\frac{n-1}{2}}}\eta(|x|r)\phi_0(r)r^{n-1}dr,\]
\[E_2(x,y)=\int_0^\infty e^{2\pi i\left( \sum_{i=1}^Ly_i\psi_i(r)\right)} \frac{1}{|rx|^{\frac{n-1}{2}}}R_{0,\sigma}(rx)\eta(|x|r)\phi_0(r)r^{n-1}dr.\]

The main terms are those $I_+$ and $I_-$. The phases we consider are thus \[\Phi_{x,y,-}(r)=-r|x|+ \sum_{i=1}^{L_1}y_i\psi_i(r)\] and \[\Phi_{x,y,+}(r)=r|x|+ \sum_{i=1}^{L_1}y_i\psi_i(r).\] We estimate the contributions of the terms $I_+(x,y)$ and $I_-(x,y)$ separately. We apply Lemma \ref{lem:vDCrad} to analyse their contributions to the $L^p$ estimates. Since the analysis is the same in either case, we only present the argument for $I_+(x,y)$, with the relevant phase denoted by $\Phi=\Phi_{x,y}=\Phi_{x,y,+}$. We use the quantity $H(x,y)=\inf_{r\leq \delta}\sum_{j=1}^{d_{L_1}}\left|\frac{\Phi_{x,y}^{(j)}(r)}{j!}\right|^\frac{1}{j}$.

To carry out the $L^p$ integration of $I_+(x,y)(1+|(y,z)|)^{-(L+\alpha)}$, we perform a dyadic partition on scales $H\sim 2^l$ and $|y|\sim 2^m$. By Lemma \ref{lem:HlowerBnd} we have $H(x,y)\gtrsim |y|^\frac{1}{d_{L_1}}$. It follows that, in $R$, $l,m\geq 0$, provided $|y|\gg 1$ with a large enough constant. When $H(x,y)\sim 2^l$, we have as a consequence of Lemma \ref{lem:vDCrad} that $|I_+(x,y)|^p\lesssim \frac{1}{|x|^{p\frac{n-1}{2}}}2^{-lp}$. Thus we see
\[\int\int\int\left|I_{+}(x,y)\right|^p\frac{1}{(1+|(y,z)|)^{p(L+\alpha)}}dxdydz\]
\[\lesssim\sum_{m\geq 0}\sum_{l\geq 0}2^{-lp}\int\int\int_{2^{l-1}<H(x,y)\leq 2^l;|x|,|z|\lesssim|y|\sim 2^m}\frac{1}{|x|^{p\frac{n-1}{2}}}\frac{1}{(1+|y|)^{p(L+\alpha)}}dxdydz\]
\[+\sum_{m\geq 0}\sum_{l\geq 0}2^{-lp}\int\int\int_{2^{l-1}<H(x,y)\leq 2^l; |x|\lesssim|y|\sim 2^m\lesssim |z|}\frac{1}{|x|^{p\frac{n-1}{2}}}\frac{1}{(1+|z|)^{p(L+\alpha)}}dxdydz\]
\[\lesssim \sum_{m\geq 0}\sum_{l\geq 0}2^{-lp}2^{-p(L+\alpha)}2^{L_1m}\int\int_{2^{l-1}<H(x,y)\leq 2^l, |x|\lesssim|y|\sim 2^m }\frac{1}{|x|^{p\frac{n-1}{2}}}dxdy\]
\[=\sum_{m\geq 0}\sum_{l\geq 0}2^{-lp}2^{-mp(L+\alpha)}2^{L_1m}J_{l,m},\]
where \[J_{l,m}=\int\int_{2^{l-1}<H(x,y)\leq 2^l; |x|\lesssim|y|\sim 2^m}\frac{1}{|x|^{p\frac{n-1}{2}}}dxdy.\]
We split the sum into two and write \[\sum_{m\geq 0}\sum_{l\geq 0}2^{-lp}2^{-mp(L+\alpha)}2^{L_1m}J_{l,m}=S_1+S_2,\]
where 
\[S_1=\sum_{l\geq 0}2^{-pl}\sum_{m\geq 2l,0}2^{-pm(L+\alpha)}2^{L_1m}J_{l,m}\]
and \[S_2=\sum_{l\geq 0}2^{-pl}\sum_{2l> m\geq 0}2^{-pm(L+\alpha)}2^{L_1m}J_{l,m}.\]

We provide two different estimates on the size of $J_{l,m}$. The first estimate is the trivial estimate \[J_{l,m}\leq\int\int_{|y|\lesssim 2^m,|x|\lesssim 2^m}\frac{1}{|x|^{p\frac{n-1}{2}}}dxdy\]
\begin{equation}\label{eq:JlmTriv}\leq 2^{mL} 2^{mn-p\left(\frac{n-1}{2}\right)m}.\end{equation}

The second estimate requires a little more care and it is here we work by the method used in \cite{arkhipovNoGaps}. We now work via polar integration.

Since $H(x,y)\sim 2^l$ we show that there exists \[r_j\in\left\lbrace \frac{1}{2^l}[-2^l\delta],\frac{1}{2^l}[-2^l\delta]+\frac{1}{2^l},\ldots,\frac{1}{2^l}[2^l\delta]-\frac{1}{2^l},\frac{1}{2^l}[2^l\delta]\right\rbrace\] such that $\left|\Phi'(r_j)\right|\lesssim 2^l$. We now choose $r_j=\frac{1}{2^l}\left[2^lr^*\right]$, where \[\sum_{j=1}^{d_{L_1}}\left|\Phi'(r^*)\right|^\frac{1}{j}=\inf_{r\in [-\delta,\delta]}\sum_{j=1}^{d_{L_1}}\left|\Phi'(r)\right|^\frac{1}{j}.\] 

 We have that \[\Phi'(r_j)=\sum_{i=0}^{d_{L_1}-1} \frac{1}{i!}\Phi^{(1+i)}(r^*)(r_j-r^*)^i.\]
We use the inequality $|r_j-r^*|\leq 2^{-l}$ and also $H(x,y)\sim 2^l$ so that $|\Phi^{(i)}(r^*)|\lesssim 2^{li}$ and \[\left|\Phi'(r_j)\right|\lesssim \sum_{i=0}^{d_{L_1}-1} 2^{l(i+1)}2^{-li}\sim 2^l.\]

Now we make the change of variables $s\mapsto \tilde{s}=\Phi'(r_j)=|x|+\sum_{i=1}^{L_1}y_i\psi_i'(r)$, where $s=|x|$ in the following polar integration. The change of variables has unit Jacobian. We denote by $e_1$ the vector $(1,0,\ldots,0)\in \RR^n$ and see that
\[J_{l,m}=\int\int\int_{2^{l-1}<H(s\omega,y)\leq 2^l, |y|\sim 2^m, s\lesssim |y|}\frac{1}{s^{p\frac{n-1}{2}}}d\sigma(\omega)s^{n-1}dsdy\]
\[\lesssim\sum_{j=[-2^l\delta]}^{[2^l\delta]}\int\int\int_{\left|\Phi_{s\omega,y}'(r_j)\right|\lesssim 2^l, |y|\lesssim 2^m, }\frac{1}{s^{p\frac{n-1}{2}}}d\sigma(\omega)s^{n-1}dsdy\]
\[\sim \sum_{j=[-2^l\delta]}^{[2^l\delta]}\int\int_{\left|\Phi_{se_1,y}'(r_j)\right|\lesssim 2^l, |y|\lesssim 2^m, }s^{n-1-p\frac{n-1}{2}}dsdy\]
\[\sim \sum_{j=[-2^l\delta]}^{[2^l\delta]}\int\int_{|\tilde{s}|\lesssim 2^l,|y|\lesssim 2^m}
\left|\tilde{s}-\sum_{i=1}^{L_1}y_i\psi_i'(r_j)\right|^{n-1-p\frac{n-1}{2}}d\tilde{s}dy\]
\[\lesssim \sum_{j=[-2^l\delta]}^{[2^l\delta]}\int\int_{|\tilde{s}|\lesssim 2^l,|y|\lesssim 2^m}\left|\tilde{s}^{n-1-p\frac{n-1}{2}}+\left|\sum_{i=1}^{L_1}y_i\psi_i'(r_j)\right|^{n-1-p\frac{n-1}{2}}\right|
d\tilde{s}dy\]
\[\lesssim 2^l \int_{|y|\lesssim2^m}\left(2^{l\left(n-p\frac{n-1}{2}\right)}+2^l2^{m\left(n-1-p\frac{n-1}{2}\right)}\right)dy
\]
\[\lesssim 2^{2l} 2^{L_1m}\left(2^{l\left(n-1-p\frac{n-1}{2}\right)}+2^{m\left(n-1-p\frac{n-1}{2}\right)}\right),
\]

We use this last estimate in the case that $2l\leq m$, in which case we see that \begin{equation}\label{eq:JlmEst2}J_{l,m}\lesssim 2^{2l} 2^{L_1m}2^{m\left(n-1-p\frac{n-1}{2}\right)}.\end{equation}

We use this bound to see that
\[S_1=\sum_{l\geq 0}2^{-pl}\sum_{m\geq 2l,0}2^{-pm(L+\alpha)}2^{L_1m}J_{l,m}\]
\[\lesssim \sum_{m\geq 0}2^{-pm(L+\alpha)}2^{Lm}2^{m\left(n-1-p\frac{n-1}{2}\right)}\sum_{0\leq l \leq \frac{m}{2}}2^{2l}2^{-pl}\]
\[\lesssim \sum_{m\geq 0}2^{-pm(L+\alpha)}2^{Lm}2^{m\left(n-1-p\frac{n-1}{2}\right)}2^{m}2^{-p\frac{m}{2}},\]
which is finite if $p>\frac{L+n}{L+\alpha+\frac{n}{2}}$.

Using the trivial bound \eqref{eq:JlmTriv} we see that
\[S_2=\sum_{l\geq 0}2^{-pl}\sum_{2l> m\geq 0}2^{-pm(L+\alpha)}2^{L_1m}J_{l,m}\] 
\[=\sum_{m\geq 0}2^{-pm(L+\alpha)}2^{mL} 2^{mn-p\left(\frac{n-1}{2}\right)}\sum_{2l> m\geq 0}2^{-pl}\] 
\[\lesssim\sum_{m\geq 0}2^{-pm(L+\alpha)}2^{mL} 2^{mn-p\left(\frac{n-1}{2}\right)}2^{-p\frac{m}{2}},\] 
which is finite if $p>\frac{L+n}{L+\alpha+\frac{n}{2}}$.

The analysis may be repeated for the contribution of $I_-(x,y)$.

It remains to consider the contribution of the error terms. For $E_2$, we have that
\[|E_2(x,y)|=\left|\int_0^\infty e^{2\pi i\left( \sum_{i=1}^Ly_i\psi_i(r)\right)} \frac{1}{|rx|^{\frac{n-1}{2}}}R_{0,\sigma}(rx)\eta(|x|r)\phi_0(r)r^{n-1}dr\right|\]
\[\leq C\frac{1}{|x|^{\frac{n-1}{2}}}\frac{1}{|x|}\int_{\frac{1}{2}|x|^{-1}}^1 r^{-1}r^{\frac{n-1}{2}}dr\sim\frac{1}{|x|^{\frac{n+1}{2}}}.\]

For $E_1$, we find
\[|E_1(x)|= \left|\int_0^\infty e^{2\pi i\sum_{i=1}^Ly_i\psi_i(r)} \widehat{\sigma}(rx)(1-\eta(|x|r))\phi_0(r)r^{n-1}dr\right|\]
\[\lesssim\int_0^{|x|^{-1}}  r^{n-1}dr\sim \frac{1}{|x|^n}.\] 

Since we know in either case that $|x|\gtrsim 1$, we see $\left|E_i(x,y)\right|\lesssim \frac{1}{|x|^\frac{n+1}{2}}$ for each $i$.

Of course we also have the bound $|E_i(x,y)|\leq C$ for each $i$. Performing polar integration, we find that \[\|E_i(x,y)(1+|(y,z)|)^{-L-\alpha}\|^p_{L^p(R,dxdydz)}\]\[\lesssim\int\int(1+|(y,z)|)^{-p(L+\alpha)}(1+|x|)^{-p\frac{n+1}{2}}dxdydz\sim \int_1^\infty r^{(n+L-1)-p(L+\alpha+\frac{n+1}{2})}dr,\] which is finite if $p>\frac{n+L}{L+\alpha+\frac{n+1}{2}}$. In particular, this is true if $p>\frac{n+L}{L+\alpha+\frac{n}{2}}$.
\end{proof}

\section{Necessity of Theorem \ref{thm:suffKrad}; proof of \eqref{eq:necKLp}} \label{sec:necMinMp}

According with Lemma \ref{lem:ajAreOne}, we consider those $\Psi$, \eqref{eq:psiHom}, where we have\begin{equation} a_j=1\mbox{, for }j\leq L_1,\,\ a_j=0,\mbox{ for }j>L_1\mbox{, and }d_1<d_2<\ldots<d_L.\end{equation}

We turn to the proof of Theorem \ref{thm:kLpNecRadial}, necessary condition for certain Bochner-Riesz kernels to be in $L^p$. We use the product expression, \eqref{eq:kerProd}, for the kernel $K_{\Gamma,\alpha}$. Since we have the lower bound $|A_{\alpha}(y,z)| \ge c_\alpha |(y,z)|^{-\alpha - L}$ for large $|(y,z)|$, it suffices to establish a lower
bound for the oscillatory integral $k(x,y)$ in a sufficiently large region. In fact we prove the following.

\begin{lemma}\label{lem:klowerEst}
The region $R$ is given as
\begin{equation}\label{eq:kLowerEst}\begin{split}
R&=\left\lbrace(x,y)\in\RR^{n+L};\frac{\delta_1}{2}\leq\frac{|x|}{|y_1|}\leq \delta_1, \frac{|y_i|}{|y_1|}\leq \delta_1\mbox{ for }i\geq 2,\vphantom{y_1\geq E, y_i\geq 1,\mbox{ and }|x|\geq 1}\right.\\
  &\quad \quad \left.\vphantom{\frac{|y_i|}{|y_1|}} y_1\geq E, y_i\geq 1\mbox{ for }i\geq 2, |x|\geq 1\right\rbrace,
\end{split}\end{equation}
the parameter $\delta_1$ is some sufficiently small constant, the parameter $E$ is some sufficiently large constant. For $(x,y)\in R$, $|k(x,y)|\gtrsim |x|^{-\frac{n}{2}}$. 
\end{lemma}

We will now show the resulting necessary conditions for $K_{\Gamma,\alpha}\in L^p$, Theorem \ref{thm:kLpNecRadial}. 

 \begin{proof}[Proof of \eqref{eq:necKLp}]
 In the case $p< 2$ we see, using the bounds on $|A_\alpha|$ from Lemma \ref{lem:Abound},
\[\int\int\int_{(x,y)\in R}|K_{\Gamma,\alpha}(x,y,z)|^pdxdydz\gtrsim\]\[ \int\int_E^{\min\lbrace E,|z|\rbrace} (1+|z|)^{-p(L+\alpha)}|\delta_1 y_1|^{L_1-1}\int_{\frac{\delta_1y_1}{2}\leq r\leq \delta_1 y_1}r^{-\frac{pn}{2}+{n-1}}drdy_1dz\]
\[+\int_E^{\infty} y_1^{-p(L+\alpha)}|\delta_1 y_1|^{L_1-1}y_1^{L_2}\int_{\frac{\delta_1y_1}{2}\leq r\leq \delta_1 y_1}r^{-\frac{pn}{2}+{n-1}}drdy_1\]
\[\sim\int_E^\infty y_1^{-p(L+\alpha)-\frac{pn}{2}+n+L-1}dy_1\]
Which is finite if and only if 
\[L-\frac{pn}{2}-p(L+\alpha)+n<0
\iff \frac{L+n}{p}-\frac{L+n}{2}-\frac{L}{2}<\alpha.\]
\end{proof}

\begin{proof}[Proof of Lemma \ref{lem:klowerEst}]
We chose our polynomial to be symmetric so we know that $d_1$ is even. This will allow us to find a unique critical point of the relevant phase function.

Using polar coordinates, we have
$$
K(x,y) = \int_0^{\infty} e^{2\pi i \Psi_y(r)} r^{n-1} \psi(r) \Bigl[
 \int_{{\mathbb S}^{n-1}} e^{2\pi i r x\cdot \omega} \, d\sigma(\omega) \Bigr] \, dr
$$
where $\Psi_y(r) = y_1 r^{d_1} + y_2 r^{d_2} + \cdots$ and the inner integral is the Fourier transform 
${\hat{\sigma}(r x)}$ of the surface measure $\sigma$ on ${\mathbb S}^{n-1}$. 
Our aim is to decompose
$k(x,y) = M(x,y) + E(x,y)$ where the error term $E$ satisfies $|E(x,y)| \le \epsilon |x|^{-n/2}$ for arbitrarily small
$\epsilon > 0$ and the main term $M$ satisfies $|M(x,y)| \gtrsim |x|^{-n/2}$ on $R$. Of course, this will prove Lemma \ref{lem:klowerEst}.

We first decompose $k= k_1 + E_1$, where 
$$
E_1(x,y) \ = \ \int_{0}^{|x|^{-1}}  e^{2\pi i \Psi_y(r)} {\hat{\sigma}(r x)} r^{n-1} \psi(r) \, dr
$$
so that 
\begin{equation}\label{2}
|E_1(x,y)| \ \lesssim \ |x|^{-n} \ \le \ \frac{\epsilon}{3} |x|^{-n/2} \ \ {\rm on} \ R.
\end{equation} 
In the case $n\ge 2$, for $k_1$ where $r \ge |x|^{-1}$,
we use the following well-known asymptotic formula (see, for example, Chapter $8$ of \cite{bigStein}); for $r|x| \ge 1$, 
$$
{\hat{\sigma}}(rx) \ = \ a \, \frac{e^{- 2\pi i r |x|}}{(r|x|)^{(n-1)/2}} \ + \ b \, \frac{e^{2\pi i r|x|}}{(r|x|)^{(n-1)/2}} \ + \ R_\sigma(r|x|)
$$
where $a, b \not= 0$ and $|R_\sigma(r|x|)| \lesssim (r|x|)^{-(n+1)/2}$. In the case $n=1$ this holds trivially with $a=b=1$ and $R_\sigma=0$, since the measure is the counting measure on $\lbrace -1,1\rbrace$. This gives a corresponding decomposition
for $k_1 = k_{1,1} + E_2 + E_3$. We have
\begin{equation}\label{1,3}
|E_3(x,y)| \  \lesssim \ |x|^{-(n+1)/2} \int_0^1 r^{(n-3)/2} dr \ \le \ \frac{\epsilon}{3} |x|^{-n/2} \ \ {\rm on} \ \  R,
\end{equation}
provided $E$ is chosen large enough, since $|x|\geq \delta_1|y_1|/2\geq \delta_1 E/2$.
Furthermore,
$$
E_2(x,y) \ = \ b \, \frac{1}{|x|^{(n-1)/2}} \, \int_{|x|^{-1}}^{\infty} e^{2\pi i\left( |x| r + \Psi_y(r)\right)} r^{(n-1)/2} \psi(r) \, dr
$$
and since $y_1 \ge 0$,
$$
| |x| + \Psi_y ' (r)  | = | |x| + k_1 y_1 r^{d_1 -1} + O(\delta y_1 r^{d_1 - 1}) | \ge |x| + d_1 y_1 r^{d_1 -1} (1 - C \delta) \ge |x|
$$
for small enough $\delta > 0$. Hence integrating by parts shows that
\begin{equation}\label{1,2}
|E_2(x,y)| \ \lesssim \ |x|^{(n+1)/2} \ \le \ \frac{\epsilon}{3} |x|^{-n/2} \ \ {\rm on} \ \ R.
\end{equation}
Finally, we have
$$
k_{1,1}(x,y) \ = \ a \, \frac{1}{|x|^{(n-1)/2}} \int_{|x|^{-1}}^{\infty} e^{2\pi i \left(- |x| r + \Psi_y(r)\right)} r^{(n-1)/2} \psi(r) dr.
$$
We make the change of variables $r = \sigma s$ where $\sigma^{d_1 - 1} = |x|/y_1 \sim \delta_1$ so that
$$
k_{1,1}(x,y) \ = \ g(x,y) \sqrt{\lambda} \int_{\lambda^{-1}}^{\infty} e^{2\pi i \lambda \Phi_{x,y}(s)} s^{(n-1)/2} \psi(\sigma s) \, ds
$$
where $\lambda = \sigma |x| \gg 1$ and $g(x,y) = a |x|^{-n/2} \sigma^{n/2} \sim_{\delta_1} |x|^{-n/2}$ on $R$. 
Here the phase is given by
$$
\Phi_{x,y}(s) \ = \ -s + s^{d_1} + \epsilon_2 s^{d_2} + \cdots + \epsilon_L s^{d_L}
$$
where $\epsilon_j = \epsilon_j(x,y) = O(\delta)$ on $R$. Hence the phase $\Phi_{x,y}(s)$ has
a unique nondegenerate critical point at $s_{*} \sim 1$ and so the classical stationary phase argument (see, for example, \cite{mythesis}) shows
that
$$
\Bigl|  \int_{\lambda^{-1}}^{\infty} e^{2\pi i \lambda \Phi_{x,y}(s)} s^{(n-1)/2} \psi(\sigma s) ds \Bigr| \ \sim \ \lambda^{-1/2},
$$
implying that 
\begin{equation}\label{1,1}
|k_{1,1}(x,y)| \ \gtrsim \ |x|^{-n/2} \ \ {\rm on} \ \ R.
\end{equation}

Hence \eqref{2}, \eqref{1,3}, \eqref{1,2} and \eqref{1,1} show that we can write $k(x,y) = M(x,y) + E(x,y)$
where $|E(x,y)| \le \epsilon |x|^{-n/2}$ and $|M(x,y)| \gtrsim |x|^{-n/2}$ on $R$ as desired.
This finishes the proof of Lemma \ref{lem:klowerEst} and hence the necessity of Theorem \ref{thm:kLpNecRadial}.
\end{proof}

\bibliographystyle{plain}
\bibliography{GBROreferences}

\begin{thebibliography}{10}

\bibitem{deLeeuw}
K.~de~Leeuw.
\newblock On lp multipliers.
\newblock {\em Annals of Math.}, 81(2), 1965.

\bibitem{drury}
S.~Drury.
\newblock Restrictions of fourier transforms to curves.
\newblock {\em Annales de l'Institut Fourier}, 35(1):117--123, 1985.

\bibitem{DMRestrDegenCurves}
S.~Drury and B.~Marshall.
\newblock Fourier restriction theorems for degenerate curves.
\newblock {\em Math. Proc. of the Cambridge Philos. Soc.}, 101(3):541--553,
  1987.

\bibitem{feffActa}
C.~Fefferman.
\newblock Inequalities for strongly singular convolution operators.
\newblock {\em Acta Math.}, 124:9--36, 1970.

\bibitem{feffBall}
C.~Fefferman.
\newblock The multiplier problem for the ball.
\newblock {\em Annals of Math.}, 94(2):330--336, 1971.

\bibitem{heoNazSee}
Y.~Heo, F.~Nazarov, and A.~Seeger.
\newblock Radial fourier multipliers in high dimensions.
\newblock {\em Acta Math.}, 206(1):55--92, 2011.

\bibitem{herz}
C.~S. Herz.
\newblock On the mean inversion of fourier and hankel transforms.
\newblock {\em Proc. Nat'l Acad. Sci. U. S. A.}, 40(10):996--999, 1954.

\bibitem{arkhipovNoGaps}
A.~Karatsuba, G.~I. Arkhipov, and V.~Chubarikov.
\newblock The index of convergence of the singular integral in tarry's problem.
\newblock {\em Doklady Math.}, 20:978--981, 1979.

\bibitem{arkhipovTrigSums}
A.~Karatsuba, G.~I. Arkhipov, and V.~Chubarikov.
\newblock {\em Trigonometric Sums In Number Theory and Analysis}, volume~39 of
  {\em De Gruyter Expositions In Math.}
\newblock Walter de Gruyter, 2008.

\bibitem{mockBR}
G.~Mockenhoupt.
\newblock {\em Singul\"{a}re Integrale vom Bochner-Riesz Type}.
\newblock PhD thesis, Universit\"{a}t Siegen, 1990.

\bibitem{bigStein}
Elias~M. Stein.
\newblock {\em Harmonic Analysis: Real-Variable Methods, Orthogonality, and
  Oscillatory Integrals.}
\newblock Princeton Math. Series. Princeton U. Press, 1993.

\bibitem{taoBRimpliesR}
T.~Tao.
\newblock The bochner-riesz conjecture implies the restriction conjecture.
\newblock {\em Duke Math. J.}, 96(2):363--375, 1999.

\bibitem{SteinWaingerProblemsInHA}
S.~Wainger and E.~M. Stein.
\newblock Problems in harmonic analysis related to curvature.
\newblock {\em Bull. Amer. Math. Soc.}, 84(6):1239--1295, 1978.

\bibitem{mythesis}
R.~Wheeler.
\newblock {\em Oscillatory integrals related to B\^ochner-Riesz multipliers}.
\newblock PhD thesis, MIGSAA, University of Edinburgh, 2020.

\bibitem{wrightPrivCom}
J.~Wright.
\newblock {Private Communication}, 2020.

\end{thebibliography}
\end{document}